\documentclass[12pt,british,reqno]{amsart}

\usepackage{babel}

\usepackage[font=small,format=plain,labelfont=bf,up,textfont=it,up]{caption}

\usepackage[utf8]{inputenc} 
\newcommand{\fecho}[1]{\ensuremath{\langle\langle #1 \rangle\rangle}}

\usepackage{cancel}

\usepackage{listings}

\usepackage{faktor}

\usepackage[T1]{fontenc}
\DeclareFontFamily{T1}{calligra}{}
\DeclareFontShape{T1}{calligra}{m}{n}{<->s*[1.44]callig15}{}

\DeclareMathAlphabet\mathrsfso      {U}{rsfso}{m}{n}

\usepackage{amsthm,amssymb,amsmath} 
\usepackage{enumerate} 
\usepackage[all]{xy}
\usepackage{mathtools}

\makeatletter
\def\@map#1#2[#3]{\mbox{$#1 \colon\thinspace #2 \longrightarrow #3$}}
\def\map#1#2{\@ifnextchar [{\@map{#1}{#2}}{\@map{#1}{#2}[#2]}}
\g@addto@macro\@floatboxreset\centering
\makeatother

\renewcommand{\epsilon}{\ensuremath{\varepsilon}}
\renewcommand{\phi}{\ensuremath{\varphi}}

\renewcommand{\to}{\ensuremath{\longrightarrow}}

\newcommand{\R}{\ensuremath{\mathbb R}}
\newcommand{\N}{\ensuremath{\mathbb N}}
\newcommand{\Z}{\ensuremath{\mathbb{Z}}}

\newcommand{\ang}[1]{\ensuremath{\left\langle #1\right\rangle}}

\newcommand{\setang}[2]{\ensuremath{\ang{#1 \,\mid\, #2}}}

\newtheoremstyle{theoremm}{}{}{\itshape}{}{\scshape}{.}{ }{}
\theoremstyle{theoremm}
\newtheorem{thm}{Theorem}
\newtheorem{lem}[thm]{Lemma}
\newtheorem{prop}[thm]{Proposition}

\newtheoremstyle{remark}{}{}{}{}{\scshape}{.}{ }{}

\theoremstyle{remark}

\newtheorem{rem}[thm]{Remark}

\numberwithin{equation}{section}
\allowdisplaybreaks

\usepackage[pdftex,%
bookmarks=true,
breaklinks=true,
bookmarksnumbered = true,
colorlinks= true,
urlcolor= green,
anchorcolor = yellow,
citecolor=blue,
]{hyperref}                   

\usepackage[x11names]{xcolor}

\begin{document}

\title[Coxeter-type quotients of surface braid groups]{Coxeter-type quotients of surface braid groups}
 
\author[R.~Diniz]{Renato Diniz}
\address{Universidade Federal do Rec\^oncavo da Bahia - CFP, Av. Nestor de Melo Pita, 535, CEP:45.300.000 - Amargosa - BA - Brasil}
\email{renatodiniz@ufrb.edu.br}

\author[O.~Ocampo]{Oscar Ocampo}
\address{Universidade Federal da Bahia, Departamento de Matem\'atica - IME, Av.~Milton Santos~S/N, CEP:~40170-110 - Salvador - BA - Brazil}
\email{oscaro@ufba.br}

\author[P.~C.~C.~Santos J\'unior]{Paulo Cesar Cerqueira dos Santos J\'unior}
\address{Secretaria da Educa\c{c}\~ao do Estado da Bahia, SEC-BA, $5^{a}$ Avenida N$^\circ 550$, centro administrativo da Bahia - CAB, CEP:~41745-004 - Salvador - BA - Brazil}
\email{pcesarmath@gmail.com}

\date{\today}


\date{\today}

\begin{abstract}
\noindent
Let $M$ be a closed surface, 
$q\geq 2$ and $n\geq 2$. In this paper, we analyze the Coxeter-type quotient group $B_n(M)(q)$ of the surface braid group $B_{n}(M)$ by the normal closure of the element $\sigma_1^q$, where $\sigma_1$ is the classic Artin generator of the Artin braid group $B_n$.
Also, we study the Coxeter-type quotient groups obtained by taking the quotient of $B_n(M)$ by the commutator subgroup of the respective pure braid group $[P_n(M),P_n(M)]$ and adding the relation $\sigma_1^q=1$, when $M$ is a  closed orientable surface or the disk. 
\end{abstract}

\maketitle

\section{Introduction}\label{sec:intro}

The braid groups of the $2$-disk, or Artin braid groups, were introduced by Artin in 1925 and further studied in 1947~\cite{A1,A2}. Surface braid groups were initially studied by Zariski~\cite{Z}, and were later generalized by Fox and Neuwirth to braid groups of arbitrary topological spaces using configuration spaces as follows~\cite{FoN}. Let $S$ be a compact, connected surface, and let $n\in \mathbb N$. The \textit{$n$th ordered configuration space of $S$}, denoted by $F_{n}(S)$, is defined by:
\begin{equation*}
F_n(S)=\left\{(x_{1},\ldots,x_{n})\in S^{n} \mid x_{i}\neq x_{j}\,\, \text{if}\,\, i\neq j;\,i,j=1,\ldots,n\right\}.
\end{equation*}
The \textit{$n$-string pure braid group $P_n(S)$ of $S$} is defined by $P_n(S)=\pi_1(F_n(S))$. The symmetric group $S_{n}$ on $n$ letters acts freely on $F_{n}(S)$ by permuting coordinates, and the \textit{$n$-string braid group $B_n(S)$ of $S$} is defined by $B_n(S)=\pi_1(F_n(S)/S_{n})$. This gives rise to the following short exact sequence:
\begin{equation}\label{eq:ses}
1 \to P_{n}(S) \to B_{n}(S) \stackrel{\sigma}{\longrightarrow} S_{n} \to 1.
\end{equation}
The map $\map{\sigma}{B_{n}(S)}[S_{n}]$ is the standard homomorphism that associates a permutation to each element of $S_{n}$. 
We note the following:
    \begin{enumerate}
        \item When $M=D^2$ is the disk then $B_n(D^2)$ (resp.\ $P_n(D^2)$) is the classical Artin braid group denoted by $B_n$ (resp.\ the classical pure Artin braid group denoted by $P_n$).
        
        \item     Follows from the definition that $F_1(S)=S$ for any surface $S$, the groups $P_1(S)$ and $B_1(S)$ are isomorphic to $\pi_1(S)$.
        For this reason, braid groups over the surface $S$ may be seen as generalizations of the fundamental group of $S$.
    \end{enumerate}

For more information on general aspects of surface braid groups we recommend \cite{Ha} and also the survey \cite{GPi}, in particular its Section~2 where equivalent definitions of these groups are given, showing different viewpoints of them. 
We recall that the Artin braid group $B_n$ admits the following presentation~\cite{A1}:
\begin{equation}\label{eq:presbn}
\bigg\langle \sigma_1, \ldots\, , \sigma_{n-1} \ \bigg\vert \ 
\begin{matrix}
\sigma_{i} \sigma_j = \sigma_j \sigma_{i} 
&\text{for} &\vert  i-j\vert > 1\\
\sigma_{i} \sigma_j \sigma_{i} = \sigma_j \sigma_{i} \sigma_j 
&\text{for} &\vert  i-j\vert  = 1
\end{matrix}
\ \bigg\rangle.
\end{equation}

It is well known that the symmetric group $S_n$ admits the following presentation:
$$
S_n=\left\langle \sigma_1,\ldots,\sigma_{n-1} \mid
\begin{array}{l}
\sigma_i\sigma_{i+1}\sigma_i=\sigma_{i+1}\sigma_i\sigma_{i+1} \textrm{ for } 1\leq i\leq n-2\\
\sigma_i\sigma_j=\sigma_j\sigma_i \textrm{ for } \left|i-j\right|\geq 2\\
\sigma_1^2=1
\end{array}
\right\rangle.
$$
Let $\fecho{g}$ denote the normal closure of an element $g$ in a group $G$.
Hence, from \eqref{eq:presbn} it is clear that $\faktor{B_n}{\fecho{ \sigma_1^2}}$ is isomorphic with $S_n$.

Let $B_n(2)=\faktor{B_n}{\fecho{ \sigma_1^2}}$. 
Notice that $B_n(2)$ is a finite group, while the braid group $B_n$ is an infinite torsion-free group. The question that naturally appears is whether the groups $B_n(k) = \faktor{B_n}{\fecho{ \sigma_1^k}}$ are finite for every $k\geq3$. 
The answer to this problem was given by Coxeter \cite{Co} using classical geometry and giving an unexpected connection between braid groups and platonic solids, see Figure~\ref{fig:platonics}, showing that $B_n(k)$ is finite if and 
only if $(k-2)(n-2)<4$, see Theorem~\ref{thm:coxeter} (see also \cite[Chapter~5, Proposition~2.2]{MK}). 

The complete statement of Coxeter's result is formulated in Subsection~\ref{sec:coxeter}. It is worth noting that it was proved differently by Assion \cite{As} using Buraus's representation of the braid groups.
Assion gave also a presentation of some symplectic groups as quotient of braid groups and it was improved by 
Wajnryb \cite{W} giving a braidlike presentation of the symplectic group $Sp(n,p)$. More recently, in \cite{BDOS} Coxeter's result is used to study the relationship between level $m$ congruence subgroups $B_n[m]$ and the normal closure of the element $\sigma_1^m$. In particular, they characterize when the normal closure of the element $\sigma_1^m$ has finite index in $B_n[m]$ and provide explicit generators for the finite quotients.

Motivated by Coxeter's work on Artin braid groups, we are interested in this problem for surface braid groups. 
From now on, let $B_{n}(M)(q)$ denote the quotient of the surface braid group $B_{n}(M)$ by the normal closure of the element $\sigma_1^q$ , where $\sigma_1$ is the classic Artin generator of the Artin braid group $B_n$ permuting the first two strands \cite{A1}. 

Our main purpose here is to study the Coxeter-type quotient of surface braid groups $B_n(M)(q)$. 
In contrast to the classical case of the disk, in this paper, we show that for every closed surface different from the sphere and the projective plane, the quotient group $B_n(M)(q)$ is infinite for all $n,q \geq 3$. 
In Subsection~\ref{subsec:kpi1} we prove the following result.

\begin{thm}
\label{thm:mainsurface}
Let $q\geq 3$ and $n\geq 2$ integers. 
Let $M$ be a closed surface 
different from the sphere and the projective plane. 
\begin{enumerate}
    \item\label{item:mainsurface1} If $M$ is orientable then the abelianization of the group  $B_n(M)(q)$ is isomorphic to ${\mathbb Z_q} \oplus H_1(M)$.
    
    \item\label{item:mainsurface2} 
    If $M$ is non-orientable then the abelianization of the group  $B_n(M)(q)$ is isomorphic to
$$    
\begin{cases}
H_1(M) & \text{if $q$ is odd},\\
{\mathbb Z_2} \oplus H_1(M) & \text{if $q$ is even}.
\end{cases}
 $$
 
    \item\label{item:mainsurface3}  For any surface $M$ different from the sphere and the projective plane, the group  $B_n(M)(q)$ is infinite.
\end{enumerate}
\end{thm}

We note that Theorem~\ref{thm:mainsurface} is also true for $q=2$. 
For instance, in \cite[P.~226]{GMP}, the authors claimed that for closed orientable surfaces, of genus $g\geq 1$, the quotient group $B_n(M)(2)$  is isomorphic to $\pi_1(M)^n\rtimes S_n$. So, it is infinite.

In Subsection~\ref{subsec:s2} we analyze the cases where $M$ is the sphere or the projective plane. We compute the abelianization of $B_n(M)(q)$ and prove the following result for sphere braid groups with few strings.

\begin{thm}
\label{thm:s2}
Let $q\geq 3$.
\begin{enumerate}
    \item $B_2(\mathbb S^2)(q)=
    \begin{cases}
    \Z_2 & \text{if $q$ is even},\\
    \{1\} & \text{if $q$ is odd}.
    \end{cases}
    $
    
    \item $B_3(\mathbb S^2)(q)\cong 
    \begin{cases}
    B_3(\mathbb S^2) & \text{if $gcd(4,q)=4$},\\
    S_3 & \text{if $gcd(4,q)=2$},\\
    \{1\} & \text{if $gcd(4,q)=1$}.
    \end{cases}
    $
    
    \item $B_4(\mathbb S^2)(q)$ is an infinite group if and only if $q\geq 6$.
\end{enumerate}
\end{thm}

Finally, in Section~\ref{sec:cryst} we show that the quotient group $\faktor{B_n(M)}{[P_n(M), P_n(M)]}(q)$ is finite when $M$ is the disk, see Theorem~\ref{Coxeimpar}, and that it is infinite when $M$ is a closed orientable surface $M$ of genus $g\geq 1$, see Proposition~\ref{prop:surfcrystcoxeter}, where $q\geq 3$, $n \geq 2$ and $[P_n(M), P_n(M)]$ is the commutator subgroup of the pure braid group of the surface $M$.

\subsection*{Acknowledgments}

The second named author would like to thank Eliane Santos, all HCA staff, Bruno Noronha, Luciano Macedo, Marcio Isabela, Andreia de Oliveira Rocha, Andreia Gracielle Santana, Ednice de Souza Santos, and Vinicius Aiala for their valuable help since July  2024, 
without whose support this work could not have been completed. 
O.O.~was partially supported by National Council for Scientific and Technological Development - CNPq through a \textit{Bolsa de Produtividade} 305422/2022-7.

\section{Coxeter-type quotients of surface braid groups}\label{sec:surfaces}

Our main purpose is to study the Coxeter-type quotient of surface braid groups $B_n(M)(q)$ obtained by considering $\sigma_1^q=1$, for $q\geq 3$ and where $\sigma_1$ is the classical Artin generator, see \cite{A1}. 
We will use presentations of surface braid groups that have in the set of generators, the Artin generators.

We start this section with the following elementary result that will be useful in this work.

\begin{lem}
\label{lem:bezout}
    Let $a$ and $b$ positive integers and let $g$ be an element in a group $G$. 
    If $g^a=1$ and $g^b=1$ then $g^d=1$, where $d=gcd(a, b)$ denote the greatest common divisor of the integers $a$ and $b$. 
\end{lem}

\begin{proof}
This result is a consequence of the Bezout's identity: If $a$ and $b$ are integers (not both $0$), then there exist integers $u$ and $v$ such that $gcd(a, b) = au + bv$, see \cite[Theorem~1.7, Section~1.2]{JJ}.
\end{proof}

\subsection{Coxeter's result for the disk}\label{sec:coxeter}

In this section we recall Coxeter's result for braid groups over the disk that strongly motivates this paper. 
Let $P$ denote one of the 5 platonic polyhedra (see Figure~\ref{fig:platonics}) and $\Sigma$ one of the faces of $P$, that corresponds to a regular polygon. 

\begin{figure}[!htb]
\begin{minipage}[b]{0.32\linewidth}
\centering
\includegraphics[width=0.3\columnwidth]{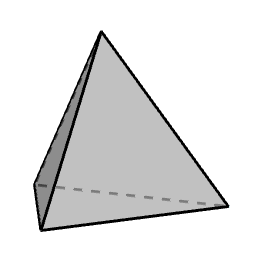}
\caption*{Tetrahedron}
\end{minipage} \hfill
\begin{minipage}[b]{0.32\linewidth}
\centering
\includegraphics[width=0.3\columnwidth]{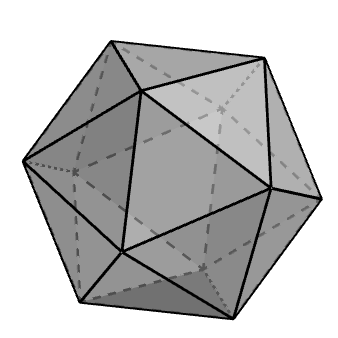}
\caption*{Icosahedron}
\end{minipage}
\begin{minipage}[b]{0.32\linewidth}
\centering
\includegraphics[width=0.3\columnwidth]{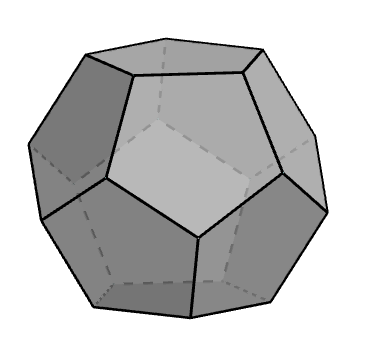}
\caption*{Dodecahedron}
\end{minipage} \hfill
\begin{minipage}[b]{0.32\linewidth}
\centering
\includegraphics[width=0.3\columnwidth]{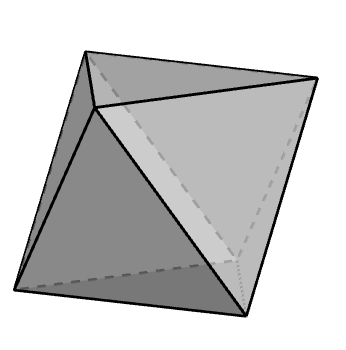}
\caption*{Octahedron}
\end{minipage} \hfill
\begin{minipage}[b]{0.32\linewidth}
\centering
\includegraphics[width=0.3\columnwidth]{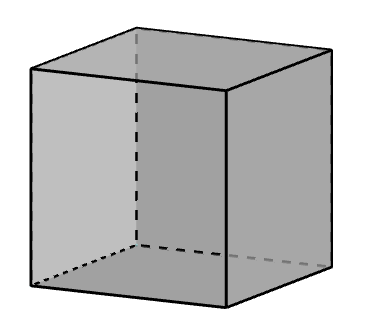}
\caption*{Cube}
\end{minipage}
\caption{The five regular polyhedra}
\label{fig:platonics}
\end{figure}

We numerically code $P$ by means of a pair of integers $(n,p)$, where
\begin{itemize}
	\item $n$ is the number of edges of $\Sigma$.
	\item $p$ is the number of polygons $\Sigma$ that meet at each vertex of $P$.
\end{itemize}
The integer pair $(n,p)$ is called the type of $P$. 
Now we state the unexpected result obtained by Coxeter about the groups $B_n(p)$.

\begin{thm}{\cite{Co}}
\label{thm:coxeter}
    Suppose $p\geq3$ and $B_n(p)$ is the quotient group derived from the $n$-braid group $B_n$ by adding one and only one relation $\sigma_1^p=1$   
$$
B_n(p)=\left\langle \sigma_1,\ldots,\sigma_{n-1} \mid
\begin{array}{l}
\sigma_i\sigma_{i+1}\sigma_i=\sigma_{i+1}\sigma_i\sigma_{i+1} \textrm{ for } 1\leq i\leq n-2\\
\sigma_i\sigma_j=\sigma_j\sigma_i \textrm{ for } \left|i-j\right|\geq 2\\
\sigma_1^p=1
\end{array}
\right\rangle.
$$

Then the quotient group $B_n(p)$ is a finite group if and only if $(n,p)$ corresponds to the type of one of the five Platonic solids (regular polyhedra).

Furthermore, the order of the finite group $B_n(p)$ is given by
$$
\left|B_n(p)\right|=\left(\frac{f}{2}\right)^{n-1} n!
$$
where $f$ is the number of faces of the Platonic solid of type $(n,p)$.
\end{thm}

Therefore, follows from Theorem~\ref{thm:coxeter} that there are only five finite groups $B_n(p)$ when $n,p\geq3$, namely:
\begin{table}[htb] 
\centering
\begin{tabular}{c|c || c|c}
\hline
Polyhedron & Type (n,p) & Quotient group & order   \\ \hline
tetrahedron & (3,3) & $B_3(3)$ & 24\\ 
hexahedron (cube) & (4,3) & $B_4(3)$ & 648\\ 
octahedron & (3,4) & $B_3(4)$ & 96 \\ 
dodecahedron & (5,3) & $B_5(3)$ & 155520 \\ 
icosahedron & (3,5) & $B_3(5)$ & 600 \\
\end{tabular}
\caption{Types of Platonic Solids and finite groups $B_n(p)$}
\end{table}

Motivated by this unexpected result from Coxeter's work on the classical braid groups, we are interested in exploring these quotients for surface braid groups, as we show in the following subsections.

\subsection{Braid groups over surfaces different from the sphere and the projective plane}
\label{subsec:kpi1}

Let $n\geq 2$ and let $B_n(M)$ denote the braid groups over a surface $M$. 
Compared with the case of the disk (see \cite{Co}) the group $B_n(M)(q)$ is infinite for any integer $q\geq 3$,  for closed surfaces different from the sphere and the projective plane. 
In this subsection we prove Theorem~\ref{thm:mainsurface}, where $H_1(M)$ the first homology group of the surface $M$. We will use presentations of surface braid groups that have in the set of generators the Artin generators. 
Given a group $G$ we denote its abelianization by $G^{Ab}$.

\begin{proof}[Proof of Theorem~\ref{thm:mainsurface}]
Let $q\geq 3$ and $n\geq 3$ integers and let $M$ be a closed surface 
different from the sphere and the projective plane.

\begin{enumerate}
    \item The proof of this item follows using a presentation of the braid group over orientable surfaces given in \cite[Theorem~1.4]{S}. Since the argument is similar for both cases (orientable and not) we give more details for the non-orientable case below.

    \item Let $M=\underbrace{\mathbb RP^2\# \cdots \# \mathbb RP^2}_{g \textrm{ projective planes}}$ where $g\geq 2$ is the genus of the non-orientable surface $M$. 
    We give a presentation of the abelianization of the group  $B_n(M)(q)$. To do this, we use the presentation of
$B_n(M)$ given by Scott \cite[Theorem~1.2]{S}: 

\begin{itemize}
	\item Generators: $\sigma_1,\ldots, \sigma_{n-1}$ and $\rho_{i,j}$ where $1\leq i\leq n$, $1\leq j\leq g$. 

\item Relations: all generators commutes. From this and using the Scott's presentation, we get the following information:  
\end{itemize}

\begin{enumerate}
    \item From \cite[Theorem~1.2, I(ii)]{S} follows $\sigma_i =\sigma_{i+1}$, for $i=1,\ldots,n-2$.

    \item From \cite[Theorem~1.2, III(ii)]{S} we get $\rho_{i,k} =\rho_{i+1,k}$, for $1\leq i\leq n-1$, $1\leq k\leq g$.

    \item In \cite[Theorem~1.2, II]{S}, were defined elements $A_{i,j}$ and $B_{i,j}$, for all $1\leq i < j\leq n$, as conjugates of $\sigma_i^2$. 
    From \cite[Theorem~1.2, II(iii)]{S} (see also \cite[Theorem~1.1, II(iii)]{S}) we obtain, for all $1\leq i < j\leq n$, $B_{i,j}=1$ in $\left( B_n(M)(q) \right)^{Ab}$.

    So, in $\left( B_n(M)(q) \right)^{Ab}$ it holds that $\sigma_i^2=1$, for all $1\leq i\leq n-1$, as well as $A_{i,j}=1$, for all $1\leq i < j\leq n$.

    \item As a consequence of the previous item and \cite[Theorem~1.2, II(i)]{S} (see also \cite[Theorem~1.1, II(i)]{S}) we get 
    $\rho_{i,g}^2\rho_{i,g-1}^2\cdots \rho_{i,1}^2 = 1$, for all $i=1,\ldots, n-1$.
\end{enumerate}
The other relations in \cite[Theorem~1.2]{S} does not contribute with further information about $\left( B_n(M)(q) \right)^{Ab}$. 

Since $\sigma_1^2=1$ and $\sigma_1^q=1$. So, from Lemma~\ref{lem:bezout}, $\sigma_1^d=1$, where $d=gcd(2,q)$.

Therefore, a presentation of the  abelianization of $B_n(M)(q)$ is given by:
\begin{itemize}
	\item Generators: $\sigma_1$ and $\rho_{1,j}$ for $1\leq j\leq g$. 

\item Relations: 
\end{itemize}

\begin{enumerate}
	
    \item all generators commutes, 

    \item $\sigma_1^2=1$, and $\sigma_1^q=1$, for $q\geq 3$. So, from Lemma~\ref{lem:bezout}, $\sigma_1^d=1$, for $q\geq 3$, where $d=gcd(2,q)$.  
		
	\item $\bf \rho_{1,g}^2\rho_{1,g-1}^2\cdots \rho_{1,1}^2   =   1$.
\end{enumerate}

We recall that a presentation of the fundamental group of the non-orientable surface $M$ of genus $g$ is given by
\begin{equation}\label{eq:presfundMnon}
\pi_1(M) = \bigg\langle \rho_{1}, \ldots , \rho_{g} 
 \ \bigg\vert \ 
\rho_{g}^2\rho_{g-1}^2\cdots \rho_{1}^2   =   1
\ \bigg\rangle.
\end{equation}

Hence, from the computations given above we proved this item
$$
\left( B_n(M)(q) \right)^{Ab} \cong  {\mathbb Z_d} \oplus H_1(M),
$$
where $d=gcd(2,q)$.

\item Since the first homology group of the closed surfaces different from the sphere and the projective plane are infinite:
$$
H_1(M)\cong 
\begin{cases}
{\mathbb Z}^{2g} & \text{if  $M$ is orientable of genus $g$}\\
{\mathbb Z}^{g-1}\oplus{\mathbb Z_2} & \text{if  $M$ is non-orientable of genus $g$}
\end{cases}
$$ 
then we conclude that the Coxeter-type quotient $B_n(M)(q)$ is infinite.
\end{enumerate}
\end{proof}

\subsection{The sphere and the projective plane}
\label{subsec:s2}

Now, we exhibit some information of $B_n(M)(q)$ when $M$ is either the sphere or the projective plane.

From \cite{FVB} we know that the sphere braid group with $n$ strings, $B_n(\mathbb S^2)$, admits a presentation with generators $\sigma_i$ for $i=1,2,\dots,n-1$ and relations as in \eqref{eq:presbn} plus:
\begin{itemize}
\item the surface relation  $\sigma_1\cdots \sigma_{n-2}\sigma_{n-1}^2\sigma_{n-2}\cdots \sigma_1=1$. 

\end{itemize}

Recall that a perfect group $G$ is a group such that $G=[G,G]$.

\begin{prop}
Let $q\geq 2$ and $n\geq 3$ integers. Let $d=gcd(q,\, 2(n-1))$. 
\begin{enumerate}
    \item The abelianization of $B_n(\mathbb S^2)(q)$ is isomorphic to the cyclic group $\mathbb Z_d$.
    \item If $q$ and $2(n-1)$ are coprimes then $B_n(\mathbb{S}^2)(q)$ is perfect.

\end{enumerate}

\end{prop}

\begin{proof}
Let $q\geq 2$ and $n\geq 3$ integers and let $d=mcd(q,\, 2(n-1))$.
Using the presentation of $B_n(\mathbb{S}^2)$ we conclude that the abelianization of the quotient group $B_n(\mathbb{S}^2)(q)$ has the presentation 
    $$
    \setang{\sigma_1}{\sigma_1^q=1,\,  \sigma_1^{2(n-1)}=1},
    $$
where the second equality comes from the surface relation. 
Lemma~\ref{lem:bezout} implies that the order of $\sigma_1\in \left(B_n(\mathbb{S}^2)(q)\right)^{Ab}$ is equal to $d$, where $d=gcd(q, 2(n-~1))$.
From this, we proved the first item.

From the first item of this result and the hypothesis of the second item, we get $\sigma_1=1$. 
Since the abelianization of $B_n(\mathbb{S}^2)(q)$ is the trivial group, then we conclude that $B_n(\mathbb{S}^2)(q)$ is perfect, proving the second item. 
\end{proof}

For the special case of few strings, in Theorem~\ref{thm:s2} we have the result for the Coxeter-type quotient of the sphere braid group, that we prove below. When analyzing the case of four strings, we use triangle groups as defined in \cite[Appendix~I, Section~7]{MK}, see also \cite{M}.

\begin{proof}[Proof of Theorem~\ref{thm:s2}]
Let $q\geq 3$.
\begin{enumerate}
    \item Since the group $B_2(\mathbb S^2)=\Z_2$ is generated by $\sigma_1$, then the result of this item follows immediately from Lemma~\ref{lem:bezout}.

    \item Recall from \cite[Third Lemma on p.248]{FVB} (see also \cite[Proposition~2.4, Chapter~11]{MK}) that $B_3(\mathbb S^2)$ has order 12 and the elements $\sigma_1$ and $\sigma_2$ have order 4. 
    So, from Lemma~\ref{lem:bezout}, in $B_3(\mathbb S^2)$ it holds
$$  \begin{cases}
    \sigma_1^4=1, & \text{if $gcd(4,q)=4$},\\
    \sigma_1^2=1, & \text{if $gcd(4,q)=2$},\\
    \sigma_1=1, & \text{if $gcd(4,q)=1$}.
    \end{cases}
$$

From this, is clear that $B_3(\mathbb S^2)(q)\cong B_3(\mathbb S^2)$ if $gcd(4,q)=4$, and that $B_3(\mathbb S^2)(q)$ is the trivial group $\{1\}$ if $gcd(4,q)=1$. 
Finally, suppose that $gcd(4,q)=2$, then it follows from the proof of \cite[Third Lemma on p.248]{FVB} (see also the proof of \cite[Proposition~2.4, Chapter~11]{MK}) that $B_3(\mathbb S^2)(q)\cong S_3$ in this last case, completing the proof of this item.

    \item The group $B_4(\mathbb S^2)(q)$ admits the following presentation:
\begin{equation}\label{eq:presb4s2}
B_4(\mathbb S^2)(q) = \bigg\langle \sigma_1, \sigma_2 , \sigma_{3} \ \bigg\vert \ 
\begin{matrix}
\sigma_1\sigma_2\sigma_1 = \sigma_2\sigma_1\sigma_2,
\sigma_2\sigma_3\sigma_2=\sigma_3\sigma_2\sigma_3,
\sigma_1\sigma_3=\sigma_3\sigma_1, \\
\sigma_1\sigma_2\sigma_3^2\sigma_2\sigma_1=1, 
\sigma_1^q=1
\end{matrix}
\ \bigg\rangle.
\end{equation}
	
We used the GAP System \cite{GAP} to show that $B_4(\mathbb S^2)(q)$ is a finite group in the following cases:
\begin{itemize}
	\item[(q=3):] The group $B_4(\mathbb S^2)(3)$ is isomorphic to the alternating group $A_4$.

	\item[(q=4):]  In this case the group $B_4(\mathbb S^2)(4)$ has order 192.

	\item[(q=5):]  The group $B_4(\mathbb S^2)(5)$ is isomorphic to the alternating group $A_5$.
\end{itemize}

We elucidate the routine used in the GAP computations for the case $B_4(\mathbb S^2)(3)$, the other cases are similar:
\begin{lstlisting}[language=GAP]
f3 := FreeGroup( "a", "b", "c" );;
gens:= GeneratorsOfGroup(f3);;
a:= gens[1];;b:= gens[2];;c:= gens[3];;
B4S23:= f3/[ a*b*a*b^-1*a^-1*b^-1, b*c*b*c^-1*b^-1*c^-1, 
a*c*a^-1*c^-1, a^3, b^3, c^3, a*b*c^2*b*a ];
Order (B4S23); 
StructureDescription (B4S23);
\end{lstlisting}

Now, for $q\geq 6$, we show that the group $B_4(\mathbb S^2)(q)$ is infinite.
Let $\langle\langle \sigma_1\sigma_3^{-1}  \rangle\rangle$ be the normal closure of the element $\sigma_1\sigma_3^{-1}$ in $B_4(\mathbb S^2)(q)$. Then
$$
B_4(\mathbb S^2)(q)/\langle\langle \sigma_1\sigma_3^{-1}  \rangle\rangle =\langle \sigma_1, \sigma_2 \mid \sigma_1\sigma_2\sigma_1 = \sigma_2\sigma_1\sigma_2,\,   (\sigma_1\sigma_2)^3=1,\, \sigma_1^q=1\rangle.
$$

Taking $a=\sigma_1\sigma_2\sigma_1$ and  $b=\sigma_1\sigma_2$ it follows that $(ab)=\sigma_1^{-1}$ and so
	$$
	B_4(\mathbb S^2)(q)/\langle\langle \sigma_1\sigma_3^{-1}  \rangle\rangle =\langle a,\, b \mid a^2=b^3=(ab)^q=1 \rangle.
	$$ 
Hence $B_4(\mathbb S^2)(q)/\langle\langle \sigma_1\sigma_3^{-1}  \rangle\rangle$ is isomorphic to the triangular group $T(2,3,q)$ that is infinite if, and only if $q\geq 6$, see \cite[Theorem~7.1,\, Appendix~I]{MK}.
\end{enumerate}
\end{proof}

Now we move to the case of the projective plane. Recall a presentation of the braid group of the projective plane.

\begin{prop}[Section~III of \cite{VB}]\label{apB_n(P2)} The braid group of the projective plane on $n$ strings, $B_n(\R P^2)$ admits the following presentation:

\item[Generators:] $\sigma_1,\sigma_2,\dots,\sigma_{n-1},\rho_1,\rho_2,\dots, \rho_n$.

\item[Relations:] 

\
\item[I] $\sigma_i\sigma_j=\sigma_j\sigma_i$ if $|i-j|\ge 2$.

\medskip

\item[II] $\sigma_i\sigma_{i+1}\sigma_i=\sigma_{i+1}\sigma_i\sigma_{i+1}$ for $i=1,\dots,n-2$.

\medskip

\item[III] $\sigma_i\rho_j=\rho_j\sigma_i$ for $j\neq i,i+1$.

\medskip

\item[IV] $\rho_i=\sigma_i\rho_{i+1}\sigma_i$ for $i=1,\dots,n-1$.

\medskip

\item[V] $\rho_{i+1}^{-1}\rho_i^{-1}\rho_{i+1}\rho_i=\sigma_i^2$.

\medskip

\item[VI] $\rho_1^2=\sigma_1\sigma_2\cdots \sigma_{n-2}\sigma_{n-1}^2\sigma_{n-2}\cdots\sigma_2\sigma_1$.
\end{prop}

For the case of braid groups over the projective plane we have the following.

\begin{prop}
Let $q\geq 2$ and $n\geq 2$ integers. 
The abelianization of the group  $B_n(\R P^2)(q)$ is isomorphic to $\Z_2$ if $q$ is odd, otherwise it is the Klein four group $\Z_2\oplus \Z_2$.
\end{prop}

\begin{proof}
We obtain the result from Lemma~\ref{lem:bezout} and the presentation of $B_n(\R P^2)$ given by Van Buskirk in \cite{VB} (see Proposition~\ref{apB_n(P2)} and also \cite[page~202, Theorem~4.1]{MK}). 
\end{proof}

\begin{rem}
   Except for the information of Theorem~\ref{thm:s2}, we do not know under which conditions on $n$ and $q$ the groups $B_n(M)(q)$ are finite, when $M$ is either the sphere or the projective plane.

\end{rem}

\section{Coxeter-type quotients and crystallographic surface braid groups}\label{sec:cryst}

The quotients of surface braid groups $B_n(M)$ by the commutator subgroup of the respective pure braid group $[P_n(M),P_n(M)]$ considered in this subsection were deeply studied in \cite{GGO} for the case of the disk and in \cite{GGOP} for the case of closed surfaces, in both cases exploring its connection with crystallographic groups.

In what follows,  we analyze the Coxeter-type quotient groups $\faktor{B_n(M)}{[P_n(M),P_n(M)]}(q)$ by adding to the presentation of $\faktor{B_n(M)}{[P_n(M), P_n(M)]}$ the relation $\sigma_1^q=1$, for braid groups over closed orientable surfaces and also for the disk.

\subsection{Braid groups over the the disk}

Unlike the case of the Coxeter quotient of the Artin braid group \cite{Co}, see Theorem~\ref{thm:coxeter}, for all $n,q \ge 3$ the Coxeter-type quotient $\faktor{B_n}{[P_n,P_n]}(q)$ is finite. 
The following result is part of the Dissertation Thesis of the third author, see \cite[Theorem~3.3]{Sa}.

\begin{thm}
\label{Coxeimpar}
Let $n,q \ge 3$ and $k\in\N$. For any integer number $q\geq 3$, the group $\faktor{B_n}{[P_n,P_n]}(q)$ is finite. 
\begin{enumerate}
	\item [(a)] If $q=2k+1$, then $\faktor{B_n}{[P_n,P_n]}(q)$ is isomorphic to $\Z_q$.
	\item [(b)] When $q=2k$, then $\faktor{B_n}{[P_n,P_n]}(q)$ has order $\frac{n(n-1)k}{2}\cdot n!$.
\end{enumerate}
    \end{thm}

\begin{proof}

Let $n,q \ge 3$ and suppose that $\sigma_1^q=1$. The integer $q$ is equal to $2k+r$, with $0\leq r\leq 1$ and $r,k\in\N$.

For item~$(a)$, as a consequence of the presentation of the Artin braid group $B_n$ given in \eqref{eq:presbn} we get $\sigma_i^{-1}\sigma_{i+1}\sigma_i=\sigma_{i+1}\sigma_i\sigma_{i+1}^{-1}$, for all $1\le i\le n-2$, and so $\sigma_i^q=1$, for all $1\le i\le n-2$. Hence, $\sigma_i=\sigma_i^{-2k}=A_{i,i+1}^{-k}$, for all $1\le i\le n-1$, where $A_{i,j}$ is an Artin generator of the  pure Artin braid group. 
So, in the group $\faktor{B_n}{[P_n,P_n]}(q)$ holds 
$[\sigma_i,\sigma_j]=1$, for all $1\le i<j\le n-1$. 
Therefore, 
\begin{align*}
\sigma_i\sigma_{i+1}\sigma_i=\sigma_{i+1}\sigma_i\sigma_{i+1}&\iff \sigma_i\sigma_{i+1}\sigma_i=\sigma_{i+1}\sigma_{i+1}\sigma_{i}\\
&\iff \sigma_i=\sigma_{i+1},
\end{align*}
for all $1\le i\le n-1$. 
Then, $\faktor{B_n}{[P_n,P_n]}(q)$ is isomorphic to $\langle \sigma_1\mid\sigma_1^p=1\rangle$, proving item~$(a)$. 

Now we prove item~$(b)$. By hypothesis, we have $\sigma_1^{2k}=1$. 
As before, we may conclude that $\sigma_i^{2k}=1$, for all $1\le i\le n$, so $A_{i,i+1}^k=1$, for all $1\le i\le n$. 
Recall the definition of the pure Artin generator $A_{i,j}=\sigma_{j-1}\sigma_{j-2}\cdots\sigma_{i}^2\cdots\sigma_{j-2}^{-1}\sigma_{j-1}^{-1}$. 
So, $A_{i,j}^k=1$, for all $1\le i<j\le n$. 
We recall that the group $\faktor{P_n}{[P_n,P_n]}$ is free abelian with a basis given by the classes of pure Artin generators $\{ A_{i,j} \mid 1\leq i<j\leq n \}$. 
Hence, in $\faktor{B_n}{[P_n,P_n]}(q)$ the natural projection of the group  $\faktor{P_n}{[P_n,P_n]}\leq \faktor{B_n}{[P_n,P_n]}$ is isomorphic to $\underbrace{\Z_k\times\cdots\times\Z_k}_\frac{n(n-1)}{2}$. 
From the above we get the following short exact sequence
$$
1\longrightarrow \underbrace{\Z_k\times\cdots\times\Z_k}_\frac{n(n-1)}{2}{\longrightarrow} \faktor{B_n}{[P_n,P_n]}(q) {\longrightarrow} S_n\longrightarrow 1.
$$
Therefore the middle group $\faktor{B_n}{[P_n,P_n]}(q)$ has finite order $\frac{n(n-1)k}{2}\cdot n!$ and with this we verify item~(b). 

From items~$(a)$ and~$(b)$ we proved that for any integer number $q\geq 3$, the group $\faktor{B_n}{[P_n,P_n]}(q)$ is finite.
\end{proof}

\subsection{Braid groups over orientable surfaces}

Let $M$ be a compact, orientable surface without boundary of genus $g\geq 1$, and let $n\geq 2$. 
We will use the presentation of $\faktor{B_n(M)}{[P_n(M), P_n(M)]}$ given in \cite{GGOP}.

\begin{prop}{\cite[Proposition~9]{GGOP}}\label{prop:pres_quo}
Let $M$ be a compact, orientable surface without boundary of genus $g\geq 1$, and let $n\geq 1$. The quotient group $\faktor{B_n(M)}{[P_n(M), P_n(M)]}$ has the following presentation:

\noindent
Generators: $\sigma_{1},\ldots, \sigma_{n-1}, a_{i,r},\,1\leq i \leq n,\, 1\leq r \leq 2g$.

\noindent
Relations:
\begin{enumerate}[(a)]
\item\label{it:pq1} the Artin relations
\begin{equation*}\label{eq:artin}
\begin{cases}
\sigma_{i}\sigma_{j}=\sigma_{j}\sigma_{i} & \text{for all $1\leq i,j\leq n-1$, $\left\lvert i-j\right\rvert \geq 2$}\\
\sigma_{i}\sigma_{i+1}\sigma_{i}=\sigma_{i+1}\sigma_{i}\sigma_{i+1} & \text{for all $1\leq i\leq n-2$.}
\end{cases}
\end{equation*}

\item\label{it:pq2} $\sigma^{2}_{i}=1$, for all $i=1,\ldots, n-1$.
	
\item\label{it:pq3} $[a_{i,r},a_{j,s}]=1$, for all $i,j=1,\ldots,n$ and $r,s=1,\ldots, 2g$.
	
\item\label{it:pq4} $\sigma_{i}a_{j,r}\sigma^{-1}_{i}= a_{\tau_{i}(j),r}$ for all $1\leq i\leq n-1$, $1\leq j\leq n$ and $1\leq r\leq 2g$. 
\end{enumerate}
\end{prop}

In \cite[Figure~9]{GM} we may see geometrically the elements $a_{i,r}$ of Proposition~\ref{prop:pres_quo}. 
We have the following result about Coxeter-type quotients and the quotient groups considered in  \cite{GGOP}. 

\begin{prop}
\label{prop:surfcrystcoxeter}
    Let $M$ be a compact, orientable surface without boundary of genus $g\geq 1$, $q\geq 3$ and let $n \geq 2$. The group $\faktor{B_n(M)}{[P_n(M), P_n(M)]}(q)$ is infinite.
\begin{enumerate}
    \item\label{it:cryst1} If $q$ is odd, the Coxeter-type quotient $\faktor{B_n(M)}{[P_n(M), P_n(M)]}(q)$ is isomorphic to a free Abelian group of rank $2g$.
    
    \item\label{it:cryst2} The group $\faktor{B_n(M)}{[P_n(M), P_n(M)]}(q)$ is isomorphic to the crystallographic group quotient $\faktor{B_n(M)}{[P_n(M), P_n(M)]}$ if $q$ is even.
\end{enumerate} 
\end{prop}

\begin{proof}
Let $M$ be a compact, orientable surface without boundary of genus $g\geq 1$, and let $n \geq 2$.
    We shall use the presentation of the quotient groups
    $\faktor{B_n(M)}{[P_n(M), P_n(M)]}$ given in Proposition~\ref{prop:pres_quo}.
    From Proposition~\ref{prop:pres_quo}\eqref{it:pq2} we have that in $\faktor{B_n(M)}{[P_n(M), P_n(M)]}(q)$ it holds  $\sigma_i^2=1$, for all $1\leq i\leq n-1$. 
    
    Hence, for all $1\leq i\leq n-1$, it is true that $\sigma_i^2=1$ and $\sigma_i^q=1$.  If $q$ is odd, then from Lemma~\ref{lem:bezout} we get $\sigma_i=1$, for all $1\leq i\leq n-1$, proving item~\eqref{it:cryst1}.  In the case that $q$ is even then, for all $1\leq i\leq n-1$, $\sigma_i^2=1$ independently of the number $q$, getting item~\eqref{it:cryst2}. 
\end{proof}

\end{document}